\documentclass[12pt]{amsart}

\newtheorem{theorem}{Theorem}[section]
\theoremstyle{plain}

\newtheorem{lemma}[theorem]{Lemma}

\newtheorem{proposition}[theorem]{Proposition}

\numberwithin{equation}{section}

\newcommand{\aut}{\mathrm{Aut}\,}			% automorphism group
\newcommand{\inn}{\mathrm{Inn}\,}			% inner mapping group
\newcommand{\ld}{\backslash}					% left division
\newcommand{\inv}{^{-1}}						% inverse
\newcommand{\gl}{\mathrm{GL}}
\def\subsp#1{\langle #1\rangle}             % subspace

\begin{document}

%\end{document}

\title{Commutative automorphic loops of order $p^3$}

\author[Barros]{Dylene Agda Souza de Barros}
\address[Barros, Grishkov]{Institute of Mathematics and Statistics, University of Sao Paulo, Rua do Mat\~{a}o, 1010, Cidade Universit\'aria, S\~{a}o Paulo, SP, Brazil, CEP 05508-090}
\email[Barros]{dylene@ime.usp.br}

\author[Grishkov]{Alexander Grishkov}
\email[Grishkov]{shuragri@gmail.com}

\author[Vojt\v{e}chovsk\'y]{Petr Vojt\v{e}chovsk\'y}
\address[Vojt\v{e}chovsk\'y]{Department of Mathematics, University of Denver, 2360 S Gaylord St, Denver, Colorado 80208, USA}
\email[Vojt\v{e}chovsk\'y]{petr@math.du.edu}

\begin{abstract}
A loop is said to be automorphic if its inner mappings are automorphisms. For a prime $p$, denote by $\mathcal A_p$ the class of all $2$-generated commutative automorphic loops $Q$ possessing a central subloop $Z\cong \mathbb Z_p$ such that $Q/Z\cong\mathbb Z_p\times\mathbb Z_p$. Upon describing the free $2$-generated nilpotent class two commutative automorphic loop and the free $2$-generated nilpotent class two commutative automorphic $p$-loop $F_p$ in the variety of loops whose elements have order dividing $p^2$ and whose associators have order dividing $p$, we show that every loop of $\mathcal A_p$ is a quotient of $F_p$ by a central subloop of order $p^3$. The automorphism group of $F_p$ induces an action of $\gl_2(p)$ on the three-dimensional subspaces of $Z(F_p)\cong (\mathbb Z_p)^4$. The orbits of this action are in one-to-one correspondence with the isomorphism classes of loops from $\mathcal A_p$. We describe the orbits, and hence we classify the loops of $\mathcal A_p$ up to isomorphism.

It is known that every commutative automorphic $p$-loop is nilpotent when $p$ is odd, and that there is a unique commutative automorphic loop of order $8$ with trivial center. Knowing $\mathcal A_p$ up to isomorphism, we easily obtain a classification of commutative automorphic loops of order $p^3$. There are precisely $7$ commutative automorphic loops of order $p^3$ for every prime $p$, including the $3$ abelian groups of order $p^3$.
\end{abstract}

\keywords{commutative automorphic loop, loops of order $p^3$, free commutative automorphic loop}

\subjclass[2010]{Primary: 20N05. Secondary: 20G40.}

\thanks{D. Barros's and A. Grishkov's stay at the University of Denver was partially supported by a grant from the Simons Foundation (grant 210176 to P. Vojt\v{e}chovsk\'y). P. Vojt\v{e}chovsk\'y thanks the Institute of Mathematics and Statistics at the University of S\~ao Paulo for hospitality and financial support. }

\maketitle

\section{Introduction}

A \emph{loop} is a set $Q$ with a binary operation $\cdot$ and a neutral element $1\in Q$ such that for every $a$, $b\in Q$ the equations $ax=b$ and $ya=b$ have unique solutions $x$, $y\in Q$, respectively. The element $x$ satisfying $ax=b$ will be denoted by $a\ld b$. See \cite{Bruck} for an introduction to the theory of loops.

Let $Q$ be a loop. For $x\in Q$, the \emph{left translation} $L_x:Q\to Q$ is defined by $L_x(y) = xy$, and the \emph{right translation} $R_x:Q\to Q$ by $R_x(y) = yx$. The \emph{inner mapping group} $\inn{Q}$ of $Q$ is the permutation group $\langle L_{x,y}$, $R_{x,y}$, $T_x;\;x$, $y\in Q\rangle$, where $L_{x,y} = L_{yx}\inv L_yL_x$, $R_{x,y} = R_{xy}\inv R_yR_x$, and $T_x = L_x\inv R_x$. Let $\aut{Q}$ be the \emph{automorphism group} of $Q$.

Denote by $Z(Q)$, $N(Q)$, $N_\lambda(Q)$ and $N_\mu(Q)$ the \emph{center}, \emph{nucleus}, \emph{left nucleus} and \emph{middle nucleus} of $Q$, respectively. For $x$, $y$, $z\in Q$, define the \emph{associator} $(x,y,z)$ of $x$, $y$, $z$ by $(xy)z = (x(yz))(x,y,z)$. The \emph{associator subloop} $A(Q)$ of $Q$ is the smallest normal subloop $H$ of $Q$ such that $Q/H$ is a group. Thus $A(Q)$ is the smallest normal subloop of $Q$ containing all associators $(x,y,z)$.

A loop $Q$ is \emph{nilpotent} if the series $Q$, $Q/Z(Q)$, $(Q/Z(Q))/Z(Q/Z(Q))$, $\dots$ terminates in $1$ in finitely many steps. In particular, $Q$ is of \emph{nilpotency class two} if $Q/Z(Q)\ne 1$ is an abelian group.

A loop $Q$ is said to be an \emph{automorphic loop} (or \emph{A-loop}) if $\inn{Q}\le\aut{Q}$. Note that a commutative loop is automorphic if and only if $L_{x,y}\in\aut{Q}$ for every $x$, $y\in Q$. This latter condition can be rewritten as
\begin{equation}\label{Eq:A}
    (yx)\ld (y(x(ab))) = [(yx)\ld(y(xa))][(yx)\ld(y(xb))],\tag{A}
\end{equation}
so a commutative loop is automorphic if and only if it satisfies the identity \eqref{Eq:A}. Groups are certainly automorphic loops, but there are many other examples.

The study of automorphic loops began with the paper \cite{BP} of Bruck and Paige. Among other results and constructions, they showed that automorphic loops are power-associative (that is, every element generates a group) and satisfy the antiautomorphic inverse property $(xy)^{-1}=y^{-1}x^{-1}$. The implicit goal of \cite{BP} was to show that diassociative (that is, every two elements generate a group) automorphic loops are Moufang. This was eventually proved by Osborn \cite{Osborn} in the commutative case, and by Kinyon, Kunen and Phillips \cite{KKP} in general.

Foundational results in the theory of commutative automorphic loops were obtained by Jedli\v{c}ka, Kinyon and Vojt\v{e}chovsk\'y in \cite{JKV1}, for instance the Odd Order Theorem, the Cauchy Theorem, and the Lagrange Theorem. In the companion paper \cite{JKV2}, the same authors noted that commutative automorphic loops of order $p$, $2p$, $4p$, $p^2$, $2p^2$ and $4p^2$ are abelian groups for every odd prime $p$, and they also constructed examples of nonassociative commutative automorphic loops of order $p^3$.

For a prime $p$, denote by $\mathcal A_p$ the class of all $2$-generated commutative automorphic loops $Q$ possessing a central subloop $Z\cong \mathbb Z_p$ such that $Q/Z\cong\mathbb Z_p\times\mathbb Z_p$. In this paper we classify the loops of $\mathcal A_p$ (cf. Theorem \ref{Th:Class2}), and also commutative automorphic loops of order $p^3$ (cf. Theorem \ref{Th:p3}) up to isomorphism. It turns out that all  loops of Theorem \ref{Th:p3} were constructed already in \cite{JKV2}, but the authors of \cite{JKV2} did not know whether their list was complete, and whether the constructed loops were pairwise nonisomorphic.

The classification of commutative automorphic loops of order $p^3$ is made possible by the fact that, when $p$ is odd, commutative automorphic $p$-loops are nilpotent, cf. \cite{JKV3}. There is a unique commutative automorphic loop of order $8$ that is not nilpotent, as can be seen quickly with a finite model builder (see \cite[Section 3]{JKV2} for details and for a near-complete human classification of commutative automorphic loops of order $8$).

In Section \ref{Sc:F} we construct the free nilpotent class two commutative automorphic loop $F$ on two generators. In Section \ref{Sc:F} we find a normal subloop $K_p$ of $F$ so that $F_p= F/K_p$ is the free nilpotent class two commutative automorphic $p$-loop on two generators in the variety of loops whose elements have order dividing $p^2$ and whose associators have order dividing $p$.

In Section \ref{Sc:Induced} we show that every loop of $\mathcal A_p$ is a quotient of $F_p$ by a central subloop of order $p^3$, and we show that $\aut F_p$ induces an action of $\gl_2(p)$ on $Z(F_p)\cong (\mathbb Z_p)^4$. Moreover, by Theorem \ref{Th:Translation}, the orbits of this action on the Grassmanian of the three-dimensional subspaces of $Z(F_p)$ correspond to the isomorphism classes of loops from $\mathcal A_p$. The orbits are described in detail in Section \ref{Sc:Orbits} (cf. Propositions \ref{Pr:n3} and \ref{Pr:3}), yielding the main results in Section \ref{Sc:Main}.

\section{The free $2$-generated commutative automorphic loop of nilpotency class two}\label{Sc:F}

Let $Z$ be an abelian group and $L$ a loop. Then a loop $Q$ is a \emph{central extension} of $Z$ by $L$ if $Z\le Z(Q)$ and $Q/Z$ is isomorphic to $L$. It is well known that $Q$ is a central extension of $Z$ by $L$ if and only if $Q$ is isomorphic to a loop $\mathcal Q(Z,L,\theta)$ defined on $L\times Z$ with multiplication
\begin{equation}\label{Eq:CentralExtension}
    (x_1,z_1)(x_2,z_2) = (x_1x_2,z_1z_2\theta(x_1,x_2)),
\end{equation}
where $\theta:L\times L\to Z$ is a \emph{(loop) cocycle}, that is, a mapping satisfying $\theta(x,1)=\theta(1,x)=1$ for every $x\in L$.

Straightforward calculation with \eqref{Eq:CentralExtension} shows that the associator in $\mathcal Q(Z,L,\theta)$ is obtained by the formula
\begin{equation}\label{Eq:Associator}
    (x_1z_1,x_2z_2,x_3z_3) = \theta(x_1,x_2)\theta(x_1x_2,x_3)\theta(x_2,x_3)^{-1}\theta(x_1,x_2x_3)^{-1}.
\end{equation}

\begin{lemma}\label{Lm:Class2}
Let $Q$ be a commutative loop of nilpotency class two. Then:
\begin{enumerate}
\item[(i)] $A(Q)\le Z(Q)$.
\item[(ii)] $(a,b,a)=1$, $(a,b,c)=(c,b,a)^{-1}$, and $(a,b,c)(b,c,a)(c,a,b)=1$ for every $a$, $b$, $c\in Q$.
\item[(iii)] $Q$ is an automorphic loop if and only if $(ab,c,d) = (a,c,d)(b,c,d)$ for every $a$, $b$, $c$, $d\in Q$.
\end{enumerate}
\end{lemma}
\begin{proof}
(i) The inclusion $A(Q)\le Z(Q)$ holds since $Q/Z(Q)$ is a group.

(ii) The identity $(a,b,a)=1$ is equivalent to $(ab)a = a(ba)$, which obviously holds in any commutative loop. Using the fact that all associators are central, we can write $(cb)a(c,b,a)^{-1} = c(ba) = (ab)c = a(bc)(a,b,c) = (cb)a(a,b,c)$, and $(a,b,c) = (c,b,a)^{-1}$ follows. Finally, $(ab)c = a(bc)(a,b,c) = (bc)a(a,b,c) = b(ca)(a,b,c)(b,c,a) = (ca)b(a,b,c)(b,c,a) = c(ab)(a,b,c)(b,c,a)(c,a,b) = (ab)c(a,b,c)(b,c,a)(c,a,b)$, hence $(a,b,c)(b,c,a)(c,a,b)=1$.

(iii) Note that $(ab)\ld(a(bc)) = c(a,b,c)^{-1}$. The identity \eqref{Eq:A} is therefore equivalent to $ab(d,c,ab)^{-1} = a(d,c,a)^{-1}b(d,c,b)^{-1}$, which is equivalent to $(ab,c,d) = (a,c,d)(b,c,d)$, by (ii).
\end{proof}

\begin{lemma}\label{Lm:AClass2}
Let $Q$ be a commutative automorphic loop of nilpotency class two. Then:
\begin{enumerate}
\item[(i)] For every $a$, $b$, $c$, $d\in Q$,
\begin{align*}
    (ab,c,d) &= (a,c,d)(b,c,d),\\
    (a,b,cd) &= (a,b,c)(a,b,d),\\
    (a,bc,d) &= (a,d,b)(a,d,c)(b,a,d)(c,a,d).
\end{align*}
\item[(ii)] For every $a$, $b$, $c$, $d\in Q$,
\begin{displaymath}
    (ab)(cd) = (ac)(bd)(ac,b,d)(b,a,c)(d,c,ab).
\end{displaymath}
\item[(iii)] For every $a$, $b\in Q$ and $i$, $j$, $k\in\mathbb Z$,
\begin{displaymath}
    (a^i,b^j,b^k) = (a,b,b)^{ijk},\quad (b^i,a^j,b^k)=1,\quad (b^i,b^j,a^k) = (b,b,a)^{ijk}.
\end{displaymath}
\item[(iv)] For every $a$, $b\in Q$ and $i_1$, $i_2$, $j_1$, $j_2$, $k_1$, $k_2\in\mathbb Z$,
\begin{displaymath}
    (a^{i_1}b^{i_2},a^{j_1}b^{j_2},a^{k_1}b^{k_2}) = (a,a,b)^{j_1(i_1k_2-i_2k_1)}(a,b,b)^{j_2(i_1k_2-i_2k_1)}.
\end{displaymath}
\end{enumerate}
\end{lemma}
\begin{proof}
(i) The first equality is from Lemma \ref{Lm:Class2}(iii). Then $(a,b,cd) = (a,b,c)(a,b,d)$ follows by $(a,b,c)=(c,b,a)^{-1}$ of Lemma \ref{Lm:Class2}(ii). Finally, by Lemma \ref{Lm:Class2}(ii), we have $(a,bc,d) = (bc,d,a)^{-1}(d,a,bc)^{-1} = (a,d,b)(a,d,c)(b,a,d)(c,a,d)$.

(ii) We have
\begin{align*}
(ab)(cd) &= ((ab)c)d(ab,c,d)^{-1} = ((ab)c)d(d,c,ab) = (c(ab))d(d,c,ab)\\
 & = ((ca)b)d(d,c,ab)(c,a,b)^{-1} =((ca)b)d(d,c,ab)(b,a,c)\\
 & = ((ac)b)d(d,c,ab)(b,a,c) = (ac)(bd)(d,c,ab)(b,a,c)(ac,b,d).
\end{align*}

(iii) We have $(b^i,a^j,b^k) = (b,a^j,b)^{ik} = 1$ by (i) and Lemma \ref{Lm:Class2}(ii). Using this fact and Lemma \ref{Lm:Class2}(ii) again, we get $(a^i,b^j,b^k) = (a,b^j,b)^{ik} = (b^j,b,a)^{-ik}(b,a,b^j)^{-ik} = (b^j,b,a)^{-ik} = (b,b,a)^{-ijk} = (a,b,b)^{ijk}$ and $(b^i,b^j,a^k) = (a^k,b^j,b^i)^{-1} = (a,b,b)^{-ijk} = (b,b,a)^{ijk}$.

(iv) Using parts (i), (ii) and (iii), we have
\begin{align*}
    (a^{i_1}b^{i_2},a^{j_1}b^{j_2},a^{k_1}b^{k_2})
    &=(a^{i_1},a^{j_1}b^{j_2},a^{k_1})(a^{i_1},a^{j_1}b^{j_2},b^{k_2})(b^{i_2},a^{j_1}b^{j_2},a^{k_1})(b^{i_2},a^{j_1}b^{j_2},b^{k_2})\\
    &=(a^{i_1},a^{j_1}b^{j_2},b^{k_2})(b^{i_2},a^{j_1}b^{j_2},a^{k_1})\\
    &=(a^{i_1},b^{k_2},a^{j_1})(a^{i_1},b^{k_2},b^{j_2})(a^{j_1},a^{i_1},b^{k_2})(b^{j_2},a^{i_1},b^{k_2})\\
    &\quad\cdot (b^{i_2},a^{k_1},a^{j_1})(b^{i_2},a^{k_1},b^{j_2})(a^{j_1},b^{i_2},a^{k_1})(b^{j_2},b^{i_2},a^{k_1})\\
    &=(a^{i_1},b^{k_2},b^{j_2})(a^{j_1},a^{i_1},b^{k_2})(b^{i_2},a^{k_1},a^{j_1})(b^{j_2},b^{i_2},a^{k_1})\\
    &=(a,b,b)^{i_1j_2k_2}(a,a,b)^{i_1j_1k_2}(b,a,a)^{i_2j_1k_1}(b,b,a)^{i_2j_2k_1}\\
    &=(a,a,b)^{j_1(i_1k_2-i_2k_1)}(a,b,b)^{j_2(i_1k_2-i_2k_1)}.
\end{align*}
\end{proof}

\begin{theorem}\label{Th:Free2}
Let $F$ be the free commutative automorphic loop of nilpotency class two with free generators $x_1$, $x_2$, and let $z_1=(x_1,x_1,x_2)$, $z_2 = (x_1,x_2,x_2)$. Then every element of $F$ can be written uniquely as $x_1^{a_1}x_2^{a_2}z_1^{a_3}z_2^{a_4}$ for some $a_1$, $a_2$, $a_3$, $a_4\in\mathbb Z$, and the multiplication in $F$ is given by
\begin{equation}\label{Eq:FreeProd}
    (x_1^{a_1}x_2^{a_2}z_1^{a_3}z_2^{a_4})(x_1^{b_1}x_2^{b_2}z_1^{b_3}z_2^{b_4})
    =x_1^{a_1+b_1}x_2^{a_2+b_2}z_1^{a_3+b_3-a_1b_1(a_2+b_2)}z_2^{a_4+b_4+a_2b_2(a_1+b_1)}.
\end{equation}
Furthermore, $Z(F) = N_\lambda(F) = N_\mu(F) = N(F) = A(F) = \langle z_1,z_2\rangle\cong \mathbb Z^2$. The loop $F$ is the central extension of the free abelian group $\langle z_1,z_2\rangle$ by the free abelian group with free generators $x_1$, $x_2$ via the cocycle
\begin{equation}\label{Eq:FreeCocycle}
    \theta(x_1^{a_1}x_2^{a_2},x_1^{b_1}x_2^{b_2}) = z_1^{-a_1b_1(a_2+b_2)}z_2^{a_2b_2(a_1+b_1)}.
\end{equation}
Finally, the associator in $F$ is given by
\begin{equation}\label{Eq:FreeAssociator}
    (x_1^{a_1}x_2^{a_2}z_1^{a_3}z_2^{a_4},x_1^{b_1}x_2^{b_2}z_1^{b_3}z_2^{b_4},x_1^{c_1}x_2^{c_2}z_1^{c_3}z_2^{c_4}) = z_1^{b_1(a_1c_2-a_2c_1)}z_2^{b_2(a_1c_2-a_2c_1)}.
\end{equation}
\end{theorem}
\begin{proof}
Consider the mapping $\widehat{\theta}:\mathbb Z^2\times\mathbb Z^2\to\mathbb Z^2$ defined by
\begin{displaymath}
    \widehat{\theta}((a_1,a_2),(b_1,b_2)) = (-a_1b_1(a_2+b_2),a_2b_2(a_1+b_1)),
\end{displaymath}
and let $\widehat{F} = \mathcal Q(\mathbb Z^2,\mathbb Z^2,\widehat{\theta})$. By \eqref{Eq:CentralExtension}, $\widehat{F}$ is defined on $\mathbb Z^4$ by
\begin{displaymath}
    (a_1,a_2,a_3,a_4)(b_1,b_2,b_3,b_4) = (a_1+b_1,a_2+b_2,a_3+b_3-a_1b_1(a_2+b_2),a_4+b_4+a_2b_2(a_1+b_1)).
\end{displaymath}
It follows that $Z(\widehat F)\ge \{(0,0,a_3,a_4);\;a_i\in\mathbb Z\}$ and $\widehat{F}$ is a commutative loop of nilpotency class at most two.

With $a=(a_1,a_2,a_3,a_4)$, $b=(b_1,b_2,b_3,b_4)$, $c=(c_1,c_2,c_3,c_4)\in\widehat{F}$, a straightforward evaluation of \eqref{Eq:Associator} yields
\begin{displaymath}
    (a,b,c) = (0,0,b_1(a_1c_2-a_2c_1),b_2(a_1c_2-a_2c_1)).
\end{displaymath}
Hence $Z(\widehat F) = \{(0,0,a_3,a_4);\;a_i\in \mathbb Z\}$. Another routine calculation gives the identity $(ab,c,d)=(a,c,d)(b,c,d)$ in $\widehat{F}$, so, by Lemma \ref{Lm:Class2}(iii), $\widehat{F}$ is an automorphic loop. With $e_1=(1,0,0,0)$, $e_2 = (0,1,0,0)$, $e_3=(0,0,1,0)$, $e_4=(0,0,0,1)$, note that $(e_1,e_1,e_2) = e_3$, $(e_1,e_2,e_2) = e_4$, and $e_1^{a_1}e_2^{a_2}e_3^{a_3}e_4^{a_4} = (a_1,a_2,a_3,a_4)$.

Let now $F$ be as in the statement of the theorem. By Lemmas \ref{Lm:Class2} and \ref{Lm:AClass2}, any word in $x_1$, $x_2$ can be written as $x_1^{a_1}x_2^{a_2}z_1^{a_3}z_2^{a_4}$, for some $a_i\in\mathbb Z$. Consider the homomorphism $\widehat{\phantom x}:F\to\widehat F$ determined by $\widehat{x_1}=e_1$, $\widehat{x_2}=e_2$. Since $\widehat{z_1} = e_3$ and $\widehat{z_2}=e_4$, we see that $\widehat{\phantom x}$ maps $x_1^{a_1}x_2^{a_2}z_1^{a_3}z_2^{a_4}$ onto $e_1^{a_1}e_2^{a_2}e_3^{a_3}e_4^{a_4} = (a_1,a_2,a_3,a_4)$. This means that $\widehat{\phantom x}:F\to\widehat{F}$ is in fact an isomorphism, and that every element of $F$ can be written uniquely as $x_1^{a_1}x_2^{a_2}z_1^{a_3}z_2^{a_4}$.

We have now established all claims of the theorem except for $Z(F) = N_\lambda(F) = N_\mu(F) = N(F)$. By Lemma \ref{Lm:AClass2}, $(x_1,x_1^{a_1}x_2^{a_2},x_2) = z_1^{a_1}z_2^{a_2}$, so $N_\mu(F)\le Z(F)$. Since $Z(Q)\le N(Q)=N_\lambda(Q)\cap N_\rho(Q)\le N_\mu(Q)$ holds in any automorphic loop $Q$ by \cite[Corollary to Lemma 2.8]{BP}, we are done.
\end{proof}

\section{The loops $K_p$ and $F_p =  F/K_p$}\label{Sc:Fp}

Let $K_p$ be the smallest normal subloop of $F$ containing
$\{a^{p^2};\;a\in F\}\cup \{a^p;\;a\in A(F)\}$.

\begin{lemma}
$K_p=\{(p^2c_1,p^2c_2,pc_3,pc_4);\;c_i\in \mathbb Z\}$.
\end{lemma}
\begin{proof}
Let $c\in F$ and $z\in Z(F)$. Since $(cz)^k = c^kz^k$, $K_p$ is generated by elements of the form $(c_1,c_2,0,0)^{p^2}$ and
$(0,0,c_3,c_4)^p$. By \eqref{Eq:FreeProd}, $(0,0,c_3,c_4)^p = (0,0,pc_3,pc_4)$. An easy
induction on $k\ge 1$ shows that
\begin{displaymath}
    (c_1,c_2,0,0)^k =
    (kc_1,kc_2,-c_1^2c_2\sum_{i=1}^{k-1}(i+i^2),c_2^2c_1\sum_{i=1}^{k-1}(i+i^2)).
\end{displaymath}
Since
\begin{displaymath}
    \sum_{i=1}^{p^2-1}(i+i^2) = (p^2-1)p^2/2 + (p^2-1)p^2(2p^2-1)/6
\end{displaymath}
is divisible by $p$, it follows that $ K_p$ is the smallest normal subloop
containing $S=\{(p^2c_1,p^2c_2,pc_3,pc_4);\;c_i\in\mathbb Z\}$.

We claim that $S= K_p$. A short calculation with \eqref{Eq:FreeProd} shows that
$S$ is closed under multiplication and division. Recall that $(ab)\ld (a(bc)) =
c(a,b,c)^{-1}$. Hence to prove that $S$ is a normal subloop it suffices to show
that $(a,b,c)\in S$ for every $a$, $b\in F$ and $c\in S$. By
\eqref{Eq:FreeAssociator}, the associator $((a_1,a_2,a_3,a_4),(b_1,b_2,b_3,b_4),(p^2c_1,p^2c_2,pc_3,pc_4))$
is equal to $(0,0,b_1(p^2a_1c_2-p^2a_2c_1),b_2(p^2a_1c_2-p^2a_2c_1))$, which is an element of $S$.
\end{proof}

As in \cite{JKV2}, for integers $0\le a$, $b<p$, define the modular overflow indicator by
\begin{displaymath}
    (a,b)_p = \left\{\begin{array}{ll}
        0,\,\text{if $a+b<p$},\\
        1,\,\text{otherwise}.
    \end{array}\right.
\end{displaymath}

Let $F_p =  F/K_p$.

\begin{lemma}\label{Lm:Fp}
$F_p$ is isomorphic to the loop defined on $(\mathbb Z_p)^6$ with multiplication
\begin{align*}
    (a_1,a_2,a_3,a_4,a_5,a_6)&(b_1,b_2,b_3,b_4,b_5,b_6)\\
    = ( &a_1+b_1,a_2+b_2,\\
        &a_3+b_3+(a_1,b_1)_p,a_4+b_4+(a_2,b_2)_p,\\
        &a_5+b_5-a_1b_1(a_2+b_2),a_6+b_6+a_2b_2(a_1+b_1)).
\end{align*}
Moreover, $Z(F_p)=N(F_p)=N_\lambda(F_p) = N_\mu(F_p) = 0\times 0 \times (\mathbb Z_p)^4$.
\end{lemma}
\begin{proof}
Note that $(a_1,a_2,a_3,a_4) K_p = (b_1,b_2,b_3,b_4) K_p$ if and only if $a_1\equiv
b_1\pmod{p^2}$, $a_2\equiv b_2\pmod{p^2}$, $a_3\equiv b_3\pmod{p}$ and
$a_4\equiv b_4\pmod{p}$. Thus $F_p$ is isomorphic to $\mathbb
Z_{p^2}\times\mathbb Z_{p^2}\times\mathbb Z_p\times\mathbb Z_p$ with
multiplication
\begin{displaymath}
    (a_1,a_2,a_3,a_4)(b_1,b_2,b_3,b_4)
    = (a_1+b_1,a_2+b_2,a_3+b_3-a_1b_1(a_2+b_2),a_4+b_4+a_2b_2(a_1+b_1)).
\end{displaymath}

For $n\in\mathbb Z_{p^2}$, write $n=n'+pn''$, where $0\le n'$, $n''<p$. Then
the addition in $\mathbb Z_{p^2}$ can be expressed on $\mathbb Z_p\times\mathbb
Z_p$ by $n+m = (n',n'')+(m',m'') = (n'+m',n''+m''+(n',m')_p)$.

If we split $a_1$, $a_2$, $b_1$, $b_2$ in this way in the above multiplication
formula, we obtain the multiplication formula
\begin{align*}
    (a_1',a_2',a_1'',a_2'',a_3,a_4)&(b_1',b_2',b_1'',b_2'',b_3,b_4)\\
    = ( &a_1'+b_1',a_2'+b_2',\\
        &a_1''+b_1''+(a_1',b_1')_p,a_2''+b_2''+(a_2',b_2')_p,\\
        &a_3+b_3-a_1'b_1'(a_2'+b_2'),a_4+b_4+a_2'b_2'(a_1'+b_1'))
\end{align*}
on $(\mathbb Z_p)^6$, since we can calculate modulo $p$ in the last two
coordinates.

We clearly have $0\times 0\times (\mathbb Z_p)^4\le Z(F_p)$. Consider
$(a_1,a_2,0,0,0,0)\ne 1$. Then
\begin{multline*}
    ((1,0,0,0,0,0)(a_1,a_2,0,0,0,0))(0,1,0,0,0,0)\\
        = (1+a_1,1+a_2,(a_1,1)_p, (a_2,1)_p, -a_1a_2,a_2(1+a_1)),
\end{multline*}
while
\begin{multline*}
    (1,0,0,0,0,0)((a_1,a_2,0,0,0,0)(0,1,0,0,0,0))\\
        = (1+a_1,1+a_2,(a_1,1)_p, (a_2,1)_p, -a_1(a_2+1),a_2a_1).
\end{multline*}
Hence $(a_1,a_2,0,0,0,0)\not\in N_\mu(F_p)$, and $Z(F_p) = N_\lambda(F_p) = N_\mu(F_p)=
N(F_p) = 0\times 0\times (\mathbb Z_p)^4$ follows.
\end{proof}

Arguing similarly to the proof of Theorem \ref{Th:Free2}, we see that $F_p$ is the free nilpotent class two $p$-loop on two generators
\begin{equation}\label{Eq:FpGens}
    x=(1,0,0,0,0,0),\quad y=(0,1,0,0,0,0)
\end{equation}
in the variety of commutative automorphic loops satisfying the identities $a^{p^2}=1$ and $(a,b,c)^p=1$.

The following symbols will be useful in expressing powers of elements in $F_p$. For $0\le a<p$ and $k\ge 1$ let
\begin{displaymath}
    [k,a]_p = (a,a)_p + (a,(2a)\ \textrm{mod}\ p)_p + \cdots + (a,((k-1)a)\ \textrm{mod}\ p)_p.
\end{displaymath}
In particular, $[1,a]_p=0$.

We now show why the case $p=3$ must be treated separately.

\begin{lemma}\label{Lm:FPowers}
For $(a_1,a_2,a_3,a_4,a_5,a_6)\in F_p$ and $k\ge 1$, $(a_1,a_2,a_3,a_4,a_5,a_6)^k$ is equal to
\begin{displaymath}
   (ka_1,ka_2,ka_3+[k,a_1]_p,ka_4+[k,a_2]_p,
    ka_5 - a_1^2a_2\sum_{i=1}^{k-1}(i+i^2),
    ka_6 + a_1a_2^2\sum_{i=1}^{k-1}(i+i^2)).
\end{displaymath}
In particular,
\begin{displaymath}
    (a_1,a_2,a_3,a_4,a_5,a_6)^p = \left\{\begin{array}{ll}
        (0,0,a_1,a_2,0,0),&\text{ if $p\ne 3$},\\
        (0,0,a_1,a_2,a_1^2a_2,-a_1a_2^2),&\text{ if $p=3$}.
    \end{array}\right.
\end{displaymath}
\end{lemma}
\begin{proof}
The general formula follows by a simple induction on $k$, using the multiplication of Lemma \ref{Lm:Fp}. Suppose that $k=p$. For $0\le a<p$ we have
\begin{displaymath}
    [p,a]_p = (a,a)_p + (a,(2a)\ \textrm{mod}\ p)_p + \cdots + (a,((p-1)a)\ \textrm{mod}\ p)_p = \sum_{i=1}^{p-1}(a,i)_p.
\end{displaymath}
Since $(a,i)_p = 1$ if and only if $a+i\ge p$, we conclude that $[p,a]_p = a$. Finally, let
\begin{displaymath}
    t_p = \sum_{i=1}^{p-1}(i+i^2) = (p-1)p/2 + (p-1)p(2p-1)/6.
\end{displaymath}
If $p>3$, we obviously have $t_p\equiv 0\pmod p$. Also, $t_2 = 1 + 1^2 = 2\equiv 0\pmod 2$ and $t_3 = 1+1^2+2+2^2 \equiv -1 \pmod 3$.
\end{proof}

\begin{lemma}\label{Lm:ZFp}
Let $x$, $y\in F_p$ be the free generators of $F_p$ from \eqref{Eq:FpGens}. Then $x^p = (0,0,1,0,0,0)$, $y^p = (0,0,0,1,0,0)$, $(x,x,y) = (0,0,0,0,1,0)$, $(x,y,y) = (0,0,0,0,0,1)$. Moreover, $Z(F_p) = 0\times 0\times \langle x^p\rangle\times \langle y^p\rangle\times \langle (x,x,y)\rangle\times \langle (x,y,y)\rangle$, $A(F_p)=0\times 0\times 0 \times 0 \times (\mathbb Z_p)^2$, and
\begin{equation}\label{Eq:Canonical}
    (a_1,a_2,a_3,a_4,a_5,a_6) = x^{a_1}y^{a_2}(x^p)^{a_3}(y^p)^{a_4}(x,x,y)^{a_5}(x,y,y)^{a_6}
\end{equation}
for every $(a_1,a_2,a_3,a_4,a_5,a_6)\in F_p$.
\end{lemma}
\begin{proof}
We have $x^p = (0,0,1,0,0,0)$, $y^p=(0,0,0,1,0,0)$ by Lemma \ref{Lm:FPowers}. A quick calculation yields $(xx)y = (2,1,(1,1)_p,0,0,0)$, $x(xy) = (2,1,(1,1)_p,0,-1,0)$, so $(x,x,y) = (0,0,0,0,1,0)$. The equality $(x,y,y) = (0,0,0,0,0,1)$ follows by a similar argument. The structure of $Z(F_p)$ and $A(F_p)$ is now clear.

For every $1\le k<p$, $[k,1]_p = (1,1)_p+\cdots +(1,k-1)_p = 0$. Therefore for every $0\le k$, $\ell<p$ we have $x^k=(k,0,0,0,0,0)$, $y^\ell = (0,\ell,0,0,0,0)$, and $x^ky^\ell = (k,\ell,0,0,0,0)$. Equation \eqref{Eq:Canonical} follows.
\end{proof}

\begin{lemma}\label{Lm:All}
If $H\le F_p$ contains $xz$, $yz'$ for some $z$, $z'\in Z(F_p)$ then $H=F_p$.
\end{lemma}
\begin{proof}
Since $(xz)^p = x^pz^p = x^p$, $(yz')^p = y^p$, $(xz,xz,yz')=(x,x,y)$ and $(xz,yz',yz') = (x,y,y)$, Lemma \ref{Lm:ZFp} implies $Z(F_p)\le H$. But then also $x$, $y\in H$ and $H=F_p$.
\end{proof}

\section{The induced action}\label{Sc:Induced}

\begin{lemma}\label{Lm:GL}
Let $\rho = \binom{\alpha_1\ \alpha_2}{\beta_1\ \beta_2}\in
\gl_2(p)$. Then $\rho$ induces an automorphism $\widehat\rho$ of $F_p$ by
\begin{equation}\label{Eq:Induced}
    \widehat\rho(x) =(\alpha_1,\alpha_2,0,0,0,0)=x^{\alpha_1}y^{\alpha_2}, \quad
    \widehat\rho(y) =(\beta_1,\beta_2,0,0,0,0)=x^{\beta_1}y^{\beta_2},
\end{equation}
where $x$, $y$ are the free generators \eqref{Eq:FpGens} of $F_p$. Moreover, if $\lambda\in\aut{F_p}$ then there are $\rho\in \gl_2(p)$ and $\sigma\in\aut{F_p}$ such that $\lambda = \widehat\rho\sigma$ and $\sigma(z)=z$ for every $z\in Z(F_p)$.
\end{lemma}
\begin{proof}
Let $\rho\in \gl_2(p)$. Since $x$, $y$ are free generators of $F_p$, the formula \eqref{Eq:Induced} correctly defines $\widehat\rho$ as an endomorphism of $F_p$. We claim that $\widehat\rho$ is an automorphism of $F_p$. Indeed, by Lemma \ref{Lm:FPowers} we have
\begin{align*}
    \widehat\rho(x)^{\beta_2}\widehat\rho(y)^{-\alpha_2}Z(F_p) &=
        (x^{\alpha_1\beta_2}y^{\alpha_2\beta_2})(x^{-\beta_1\alpha_2}y^{-\beta_2\alpha_2})Z(F_p)
        = x^{\det\rho}Z(F_p),\\
    \widehat\rho(x)^{-\beta_1}\widehat\rho(y)^{\alpha_1}Z(F_p) &=
        (x^{-\alpha_1\beta_1}y^{-\alpha_2\beta_1})(x^{\beta_1\alpha_1}y^{\beta_2\alpha_1})Z(F_p)
        = y^{\det\rho}Z(F_p).
\end{align*}
Thus $xz$, $yz'\in \langle\widehat\rho(x),\widehat\rho(y)\rangle$ for some $z$, $z'\in Z(F_p)$, and $\widehat\rho$ is onto $F_p$ by Lemma \ref{Lm:All}.

Now let $\lambda\in\aut{F_p}$, where $\lambda(x)=x^{\alpha_1}y^{\alpha_2}z_x$, $\lambda(y)=x^{\beta_1}y^{\beta_2}z_y$ for some $0\le \alpha_1$, $\alpha_2$, $\beta_1$, $\beta_2<p$ and $z_x$, $z_y\in Z(F_p)$. Then $\lambda$ induces an automorphism of $F_p/Z(F_p)\cong \mathbb Z_p\times \mathbb Z_p$ by $xZ(F_p)\mapsto (\alpha_1,\alpha_2)$ and $yZ(F_p)\mapsto (\beta_1,\beta_2)$, which means that $\rho = \binom{\alpha_1\ \alpha_2}{\beta_1\ \beta_2}$ belongs to $\gl_2(p)$, and we can consider the induced automorphism $\widehat\rho\in\aut{F_p}$.

Let $\tau = \lambda^{-1}\widehat\rho\in\aut{F_p}$. Then $\tau(x) = \lambda^{-1}\widehat\rho(x) = \lambda^{-1}(\widehat\rho(x)z_xz_x^{-1}) = \lambda^{-1}(\lambda(x)z_x^{-1}) = xz$, where $z=\lambda^{-1}(z_x^{-1})\in Z(F_p)$. Similarly, $\tau(y) = yz'$ for some $z'\in Z(F_p)$. Then $\tau(x^p) = \tau(x)^p = (xz)^p = x^p$, $\tau(y^p) = y^p$, and $\tau((x,x,y)) = (\tau(x),\tau(x),\tau(y)) = (xz,xz,yz') = (x,x,y)$, $\tau((x,y,y)) = (x,y,y)$. By Lemma \ref{Lm:ZFp}, $\tau$ is identical on $Z(F_p)$.
\end{proof}

Recall that $\mathcal A_p$ is the class of all $2$-generated commutative automorphic loops $Q$ possessing a central subloop $Z\cong \mathbb Z_p$ such that $Q/Z\cong\mathbb Z_p\times\mathbb Z_p$. Note that among the three abelian groups $(\mathbb Z_p)^3$, $\mathbb Z_p\times \mathbb Z_{p^2}$ and $\mathbb Z_{p^3}$ of order $p^3$ only $\mathbb Z_p\times \mathbb Z_{p^2}$ belongs to $\mathcal A_p$.

\begin{theorem}\label{Th:Translation}
Let $p$ be a prime.
\begin{enumerate}
\item[(i)] Let $Q\in\mathcal A_p$. Then there is an epimorphism $\varphi:F_p\to Q$ with $\ker\varphi\le Z(F_p)$.

\item[(ii)] Let $Q_1$, $Q_2\in\mathcal A_p$, and let $\varphi_i:F_p\to Q_i$ be the induced epimorphisms with $\ker(\varphi_i)\le Z(F_p)$.
    Then $Q_1\cong Q_2$ if and only if there is $\rho\in \gl_2(p)$ such that $\ker\varphi_2=\widehat\rho(\ker\varphi_1)$.
\end{enumerate}
\end{theorem}
\begin{proof}

(i) Let $a$, $b\in Q$ be such that $\langle a,b\rangle = Q$. Since $Q$ is $2$-generated and of nilpotency class at most two, there is an epimorphism $\psi: F\to Q$ such that $\psi(x_1)=a$, $\psi(x_2)=b$, using the notation of Theorem \ref{Th:Free2}.

We have $c^{p^2}=1$ for every $c\in Q$ (else $Q$ is cyclic), and $c^p=1$ for every $c\in A(Q)$ (because $Q/Z$ is a group and so $A(Q)\le Z\cong \mathbb Z_p$). We have just shown that $\psi(c)=1$ for every generator $c$ of $K_p$, and thus $K_p\le\ker(\psi)$. Let $\varphi:F_p\to Q$ be the epimorphism induced by $\psi$, that is, $\varphi(cK_p) = \psi(c)$.

Suppose, for a contradiction, that $\ker\varphi$ is not contained in $Z(F_p)$. Then there is $x^iy^jz\in \ker\varphi\setminus Z(F_p)$ for some $0\le i$, $j<p$ and $z\in Z(F_p)$, where we can assume without loss of generality that $i>0$. Let $H=\langle x^iy^jz,y\rangle$. Then $x^iz'\in H$ for some $z'\in Z(F_p)$, and so $xz''\in H$ for some $z''\in Z(F_p)$. By Lemma \ref{Lm:All}, $H=F_p$. But then $Q=\varphi(F_p) = \varphi(\langle x^iy^jz,y\rangle) = \varphi(\langle y\rangle)$ is cyclic, a contradiction.

(ii) Suppose that $\kappa:Q_1\to Q_2$ is an isomorphism and consider the diagram \eqref{d}, where $N_i=\ker\varphi_i$.

\begin{equation}\label{d}
\begin{array}{lllllllll}
1 & \to & N_2 & \to &F_p& \stackrel{\varphi_2}{\rightarrow} & Q_2 &\to & 1\\
 &  & \uparrow \mu &  & \uparrow \lambda &  & \uparrow \kappa &  &  \\
1 & \to & N_1 & \to &F_p& \stackrel{\varphi_1}{\rightarrow}  & Q_1 & \to & 1.
\end{array}
\end{equation}

Since $\varphi_2$ is onto $Q_2$, there are $x'$, $y'\in F_p$ such that $\varphi_2(x') = \kappa\varphi_1(x)$ and $\varphi_2(y') = \kappa\varphi_1(y)$. As $F_p$ has free generators $x$, $y$, an endomorphism $\lambda:F_p\to F_p$ is determined by the values $\lambda(x)=x'$, $\lambda(y)=y'$, and we have $\varphi_2\lambda = \kappa\varphi_1$. Moreover, $\langle x',y'\rangle$ intersects every coset of $N_2$ in $F_p$, else $\varphi_2\lambda = \kappa\varphi_1$ is not onto $Q_2$. Then $\langle x',y'\rangle$ also intersects every coset of $Z(F_p)\ge N_2$ in $F_p$, in particular the cosets $xZ(F_p)$, $yZ(F_p)$. By Lemma \ref{Lm:All}, $\lambda\in\aut F_p$.

Let $\mu$ be the restriction of $\lambda$ to $N_1$. Then for $n\in N_1$ we have $\varphi_2\lambda(n) = \kappa\varphi_1(n) = \kappa(1)=1$. Thus $\mu(n)=\lambda(n)\in\ker(\varphi_2)=N_2$, and $\mu$ is a monomorphism $N_1\to N_2$. Since $|F_p/N_i|=|Q_i|$, $Q_1\cong Q_2$, and $F_p$ is finite, it follows that $|N_1|=|N_2|$ and $\mu:N_1\to N_2$ is an isomorphism. By Lemma \ref{Lm:GL}, we can write $\lambda = \widehat\rho\sigma$ for some $\rho\in \gl_2(p)$ and $\sigma\in\aut{F_p}$ such that $\sigma(z)=z$ for every $z\in Z(F_p)$. Since $N_1\le Z(F_p)$ by (i), it follows that $N_2 = \mu(N_1) = \lambda(N_1) = \widehat\rho\sigma(N_1) = \widehat\rho(N_1)$.

Conversely, suppose there is $\rho\in\gl_2(p)$ such that $\widehat\rho(N_1)=N_2$. Define $\kappa:Q_1\to Q_2$ by $\kappa(\varphi_1(u)) = \varphi_2\widehat\rho(u)$. This correctly defines $\kappa$ because $\varphi_1$ is onto $Q_1$, and if $\varphi_1(u)=\varphi_1(v)$ then $uN_1 = vN_1$, $\widehat\rho(u)N_2 = \widehat\rho(v)N_2$, and $\varphi_2\widehat\rho(u)=\varphi_2\widehat\rho(v)$. Moreover, $\kappa$ is a homomorphism (since $\varphi_1$, $\varphi_2$, $\widehat\rho$ are homomorphisms), it is onto $Q_2$ (since $\widehat\rho$, $\varphi_2$ are onto), and it is one-to-one (since $\kappa(\varphi_1(u))=1$ implies $\varphi_2\widehat\rho(u)=1$, $\widehat\rho(u)\in N_2$, $u\in N_1$, $\varphi_1(u)=1$).
\end{proof}

Note that $|\ker\varphi| = |F_p|/|Q| = p^6/p^3= p^3$ in Theorem \ref{Th:Translation}.

\section{The orbits}\label{Sc:Orbits}

By Theorem \ref{Th:Translation}, in order to classify the loops of $\mathcal A_p$ up to isomorphism, it suffices to describe the orbits of the action of $\gl_2(p)$ from Lemma \ref{Lm:GL} on $3$-dimensional subspaces of $Z(F_p)\cong (\mathbb Z_p)^4$.

From now on, we will write $Z(F_p) =  \subsp{x^p}\oplus \subsp{y^p}\oplus\subsp{(x,x,y)}\oplus \subsp{(x,y,y)}$ additively, and we will not distinguish between $\rho\in \gl_2(p)$ and the induced automorphism $\widehat{\rho}$ of $F_p$. Moreover, we will write $\rho(x,x,y)$ instead of the formally correct $\rho((x,x,y))$.

For the rest of this section, let $G=\gl_2(p)$, $V=\subsp{x^p}\oplus\subsp{y^p}$, and $W=\subsp{(x,x,y)}\oplus\subsp{(x,y,y)}$.

\begin{lemma}\label{Lm:Action}
The action of $G$ on $Z(F_p)$ from Lemma \ref{Lm:GL} is given by
\begin{align*}
    \rho(x^p) &=\left\{\begin{array}{ll}
        \alpha_1x^p + \alpha_2y^p,&\text{ if $p\ne 3$},\\
        \alpha_1x^p + \alpha_2y^p + \alpha_1^2\alpha_2(x,x,y) - \alpha_1\alpha_2^2(x,y,y),&\text{ if $p=3$},
    \end{array}\right.\\
    \rho(y^p) & = \left\{\begin{array}{ll}
        \beta_1x^p + \beta_2y^p,&\text{ if $p\ne 3$},\\
        \beta_1x^p + \beta_2y^p + \beta_1^2\beta_2(x,x,y) - \beta_1\beta_2^2(x,y,y),&\text{ if $p=3$},
    \end{array}\right.\\
    \rho(x,x,y) &= \alpha_1\det \rho\;(x,x,y) + \alpha_2\det \rho\;(x,y,y),\\
    \rho(x,y,y) &= \beta_1\det \rho\;(x,x,y) + \beta_2\det \rho\;(x,y,y),
\end{align*}
where $\rho = \binom{\alpha_1\ \alpha_2}{\beta_1\ \beta_2}\in\gl_2(p)$.

In particular, $W$ is always an invariant subspace, and $V$ is an invariant subspace if $p\ne 3$.
\end{lemma}
\begin{proof}
We have $\rho(x^p) = \rho(x)^p = (\alpha_1,\alpha_2,0,0,0,0)^p$ by definition. By Lemma \ref{Lm:FPowers}, $(\alpha_1,\alpha_2,0,0,0,0)^p$ equals $(0,0,\alpha_1,\alpha_2,0,0)$ when $p\ne 3$, and $(0,0,\alpha_1,\alpha_2,\alpha_1^2\alpha_2,-\alpha_1\alpha_2^2)$ when $p=3$. Similarly for $\rho(y^p)$. Calculating in $F_p$, we have
\begin{align*}
    \rho(x,x,y) &= (\rho(x),\rho(x),\rho(y)) = (x^{\alpha_1}y^{\alpha_2}, x^{\alpha_1}y^{\alpha_2}, x^{\beta_1}y^{\beta_2})\\
    &= (x,x,y)^{\alpha_1(\alpha_1\beta_2-\alpha_2\beta_1)}(x,x,y)^{\alpha_2(\alpha_1\beta_2-\alpha_2\beta_1)}
\end{align*}
by Lemma \ref{Lm:AClass2}(iv), and so $\rho(x,x,y) = \alpha_1\det\rho\;(x,x,y)+\alpha_2\det\rho\;(x,y,y)$ in the additive notation of $Z(F_p)$. Similarly for $\rho(x,y,y)$.
\end{proof}

Given a subspace $N$ of $Z(F_p)$, denote by $G(N)$ the orbit of $N$ under the action of $G$.

\begin{lemma}\label{Lm:Aux}
Let $N$ be a $3$-dimensional subspace of $Z(F_p)$, and assume that $p\ne 3$.
\begin{enumerate}
\item[(i)] If $W\not\subseteq N$ then $N\in G(\subsp{v_1,v_2+w_2,(x,x,y)})$ for some $v_1$, $v_2\in V$ and $w_2\in W$.
\item[(ii)] If $V\not\subseteq N$ then $N\in G(\subsp{x^p,v_2+w_2,w_3})$ for some $v_2\in V$ and $w_2$, $w_3\in W$.
\end{enumerate}
\end{lemma}
\begin{proof}
(i) Since $\dim N = 3$ and $\dim W=2$, there is $w\in W$ such that $N\cap W = \subsp{w}$. The subspace $W$ is invariant by Lemma \ref{Lm:Action}, so there is $\rho\in G$ such that $\rho(w) = (x,x,y)$. Since $\dim(N\cap V)\ge 1$ and $G(V)=V$, there is $v_1\in V\cap\rho(N)$. Any element of $Z(F_p)$ can be written as $v_2+w_2$ for some $v_2\in V$, $w_2\in W$. Part (ii) is similar.
\end{proof}

\begin{proposition}\label{Pr:n3}
Assume that $p\ne 3$. Let
\begin{align*}
    O_1 &= G(\subsp{x^p}\oplus W),\\
    O_2 &= G(V\oplus\subsp{(x,x,y)}),\\
    O_3 &= G(\subsp{x^p,y^p+(x,y,y),(x,x,y)}),\\
    O_4 &= G(\subsp{y^p,x^p+(x,y,y),(x,x,y)}),\\
    O_5 &= G(\subsp{y^p,\lambda x^p+(x,y,y),(x,x,y)}),
\end{align*}
where $\lambda$ is not a square in $\mathbb Z_p$. (When $p=2$, every element of $\mathbb Z_p$ is a square, and we let $O_5=\emptyset$.) Then $O_1\cup O_2\cup O_3\cup O_4\cup O_5$ is a disjoint union of all $3$-dimensional subspaces of $Z(F_p)$.
\end{proposition}
\begin{proof}
Let $N$ be a $3$-dimensional subspace of $Z(F_p)$. Throughout the proof, assume that all elements $v_i$ belong to $V$, and all elements $w_i$ belong to $W$.

First suppose that $W\subseteq N$. Then $V\not\subseteq N$ (else $Z(F_p)= V\oplus W \subseteq N$, a contradiction), so Lemma \ref{Lm:Aux} implies that $N \in G(\subsp{x^p,v_2+w_2,w_3})$. Since $G(W)=W\subseteq N$, we have $W\subseteq\subsp{x^p,v_2+w_2,w_3}$, so $\subsp{x^p,v_2+w_2,w_3} = \subsp{x^p,v_2}\oplus W$. Thus $v_2\in\subsp{x^p}$ and $N\in O_1$.

Now suppose that $V\subseteq N$. Then $W\not\subseteq N$ and Lemma \ref{Lm:Aux} implies that $N \in G(\subsp{v_1,v_2+w_2,(x,x,y)})$. Since $V\subseteq N$ and $G(V)=V$, we deduce $N\in G(V\oplus\subsp{(x,x,y)}) = O_2$. Conversely, any element of $O_2$ contains $V$.

Finally suppose that $V\not\subseteq N$ and $W\not\subseteq N$. By Lemma \ref{Lm:Aux}, $N\in G(\subsp{v_1,v_2+w_2,(x,x,y)})$. Note that $\subsp{v_1,v_2}=V$, else $\subsp{v_1,v_2+w_2,(x,x,y)} = \subsp{v_1,w_2,(x,x,y)}$, and either this subspace contains $W$ (when $w_2\not\in\subsp{(x,x,y)}$), a contradiction, or it has dimension $2$ (when $w_2\in\subsp{(x,x,y)}$), a contradiction again. Also,
$\subsp{w_2,(x,x,y)}=W$, else $\subsp{v_1,v_2+w_2,(x,x,y)} = \subsp{v_1,v_2,(x,x,y)}$, $V\subseteq N$, a contradiction. We can therefore assume that $N\in G(\subsp{v_1,v_2+\gamma(x,y,y),(x,x,y)})$, where $\gamma\ne 0$, $\subsp{v_1,v_2}=V$ and $v_1 = \tau_1x^p+\tau_2y^p$ for some $\tau_1$, $\tau_2$.

Suppose first that $\tau_2=0$. Then $N\in G(\subsp{x^p,y^p+\gamma(x,y,y),(x,x,y)})$ for some $\gamma\ne 0$. Consider $\rho = \binom{\alpha_1\ 0}{0\ \beta_2}\in G$, where $\beta_2 = (\alpha_1\gamma)^{-1}$. Then $\rho(x^p)=\alpha_1x^p$, $\rho(y^p) = \beta_2y^p$, $\rho(x,x,y) = \alpha_1^2\beta_2(x,x,y)$ and $\rho(y^p+\gamma(x,x,y)) = \beta_2y^p + \gamma\alpha_1\beta_2^2(x,y,y) = \beta_2(y^p+(x,y,y))$. Hence $N\in O_3$.

Now suppose that $\tau_2\ne 0$. Consider $\rho=\binom{\alpha_1\ 0}{\beta_1\ \beta_2}\in G$, where $\beta_1 = -\tau_1\alpha_1/\tau_2$. Then $\rho(x,x,y) = \alpha_1^2\beta_2(x,x,y)$, $\rho(x,y,y) = \beta_1\alpha_1\beta_2(x,x,y) + \alpha_1\beta_2^2(x,x,y)$, and $\rho(v_1) = \tau_1\alpha_1x^p + \tau_2(\beta_1x^p+\beta_2y^p) = \tau_2\beta_2y^p$. Hence we have $N\in G(\subsp{y^p,x^p+\gamma(x,y,y),(x,x,y)})$ for some $\gamma\ne 0$. Consider again $\rho = \binom{\alpha_1\ 0}{0\ \beta_2}\in G$. Then $\rho(\subsp{y^p,x^p+\gamma(x,y,y),(x,x,y)}) = \subsp{y^p,\alpha_1x^p+\gamma\alpha_1\beta_2^2(x,y,y),(x,x,y)} = \subsp{y^p,x^p+\gamma\beta_2^2(x,y,y),(x,x,y)}$, and we conclude that either $N\in O_4$ or $N\in O_5$, depending on whether $\gamma$ is a square in $\mathbb Z_p$.

It remains to show that $O_1$, $\dots$, $O_5$ are pairwise disjoint. If $N\in O_1$, we have $W\subseteq N$, and thus $O_1\cap O_i=\emptyset$ for $i>1$. If $N\in O_2$, we have $V\subseteq N$, and thus $O_2\cap O_i=\emptyset$ for $i>2$.

Suppose that $N\in O_3\cap O_4$. Then there is $\rho=\binom{\alpha_1\ \alpha_2}{\beta_1\ \beta_2}\in G$ such that, without loss of generality,  $N = \subsp{y^p,x^p+(x,y,y),(x,x,y)} = \rho(\subsp{x^p,y^p+(x,y,y),(x,x,y)}) = \subsp{\alpha_1x^p + \alpha_2y^p, \beta_1x^p+\beta_2y^p + w_1,w_2}$. Since $y^p\in N$, we conclude that $\alpha_1x^p\in N$. Should $\alpha_1\ne 0$, we would have $V\subseteq N$, a contradiction. Hence $\alpha_1=0$ and $\rho(x,x,y) = \alpha_2\det\rho\;(x,y,y)\in N$. As $\alpha_2\ne 0$, we see that $(x,y,y)\in N$, $x^p\in N$, $V\subseteq N$, a contradiction again. Therefore $O_3\cap O_4=\emptyset$. Similarly, $O_3\cap O_5=\emptyset$.

Suppose that $N\in O_4\cap O_5$. We can assume that $N=\subsp{y^p,\lambda x^p+(x,y,y),(x,x,y)} = \rho(\subsp{y^p,x^p+(x,y,y),(x,x,y)}) = \subsp{\beta_1x^p+\beta_2y^p,v_1+w_1,w_2}$, where $\lambda$ is not a square. If $\beta_1\ne 0$, we conclude that $x^p\in N$, $V\subseteq N$, a contradiction. Thus $\beta_1=0$, $\det\rho = \alpha_1\beta_2$, and $\alpha_1x^p + \alpha_2y^p + \alpha_1\beta_2^2(x,y,y)\in N$, $\alpha_1x^p+\alpha_1\beta_2^2(x,y,y)\in N$, $\lambda x^p + \lambda\beta_2^2(x,y,y)\in N$, and since also $\lambda x^p+(x,y,y)\in N$, we have $(\lambda\beta_2^2-1)(x,y,y)\in N$. As $\lambda$ is not a square, $\lambda\beta_2^2$ is not a square, but $1$ is a square, thus $\lambda\beta_2^2-1\ne 0$, $(x,y,y)\in N$, $W\subseteq N$, a contradiction.
\end{proof}

\begin{proposition}\label{Pr:3}
Suppose that $p=3$, and let
\begin{align*}
    O_1 &= G(\subsp{x^p}\oplus W),\\
    O_2 &= G(V\oplus\subsp{(x,x,y)}),\\
    O_3 &= G(\subsp{x^p+(x,y,y),y^p,(x,x,y)}),\\
    O_4 &= G(\subsp{x^p-(x,y,y),y^p,(x,x,y)}),\\
    O_5 &= G(\subsp{x^p-(x,y,y),y^p+(x,y,y),(x,x,y)}).
\end{align*}
Then $O_1\cup O_2\cup O_3\cup O_4\cup O_5$ is a disjoint union of all $3$-dimensional subspaces of $Z(F_p)$.
\end{proposition}
\begin{proof}
We leave the proof to the reader, who will need Lemma \ref{Lm:Action} and an argument similar to the proof of Proposition \ref{Pr:n3}. Alternatively, the proof can be accomplished by a direct computer calculation, for instance in \texttt{GAP}. (The calculation will in addition show that the cardinalities of $O_1$, $\dots$, $O_5$ are $12$, $4$, $12$, $4$ and $8$, respectively, for the correct total of $40 = (3^4-1)(3^4-3)(3^4-3^2)/((3^3-1)(3^3-3)(3^3-3^2))$ subspaces.)
\end{proof}

\section{Main result}\label{Sc:Main}

\begin{theorem}\label{Th:Class2}
Let $p$ be a prime, and let $\mathcal A_p$ be the class of all $2$-generated commutative automorphic loops $Q$ possessing a central subloop $Z\cong \mathbb Z_p$ such that $Q/Z\cong\mathbb Z_p\times\mathbb Z_p$. If $p=2$, let $O_1$, $\dots$, $O_4$ be as in Proposition \ref{Pr:n3}. If $p=3$, let $O_1$, $\dots$, $O_5$ be as in Proposition \ref{Pr:3}. If $p>3$, let $O_1$, $\dots$, $O_5$ be as in Proposition \ref{Pr:n3}. For every $i$, let $N_i\in O_i$ and $Q_i=F_p/N_i$. Then $Q\in\mathcal A_p$ is isomorphic to precisely one $Q_i$. Moreover, $Q_i$ is a group if and only if $i=1$.
\end{theorem}
\begin{proof}
Combine Theorem \ref{Th:Translation} and Propositions \ref{Pr:n3}, \ref{Pr:3}. It remains to show that $Q_i$ is a group if and only if $i=1$. Now, $Q_i=F_p/N_i$ is a group if and only if $A(F_p)\le N_i$. By Lemma \ref{Lm:ZFp}, $A(F_p)=W$. A quick inspection of the orbits shows that only $N_1$ contains $W$.
\end{proof}

We conclude the paper with a classification of commmutative automorphic loops of order $p^3$.

\begin{lemma}\label{Lm:InAp}
Let $p$ be a prime and let $Q$ be a nilpotent commutative automorphic loop of order $p^3$. If $Q$ is not a group then $Q\in\mathcal A_p$.
\end{lemma}
\begin{proof}
The center of $Q$ is a nontrivial normal associative subloop of $Q$, hence of order dividing $p^3$. Let $Z$ be a central subloop of $Q$ of order $p$, so $Z\cong\mathbb Z_p$. Then $Q/Z$ is a commutative automorphic loop of order $p^2$, necessarily a group by \cite{JKV2}. Since $Q$ is power-associative, $Q/Z$ cannot be cyclic, else $Q$ is associative. Hence $Q/Z\cong \mathbb Z_p\times\mathbb Z_p$.

Assume for a while that $Q$ is not $2$-generated. Then any two elements of $Q$ generate a proper subloop of $Q$, hence a group of order at most $p^2$. Thus $Q$ is diassociative. By the already-mentioned result of Osborn \cite{Osborn}, $Q$ is a commutative Moufang loop. But commutative Moufang loops of order $p^3$ are associative by \cite{Bruck}, a contradiction. Hence $Q$ is $2$-generated.
\end{proof}

The only commutative automorphic loop of order $8$ with trivial center has been described in \cite[Section 3]{JKV2}. Its multiplication table is
\begin{equation}\label{Eq:8}
\begin{array}{c|cccccccc}
    &1&2&3&4&5&6&7&8\\
    \hline
    1&1&2&3&4&5&6&7&8\\
    2&2&1&4&3&6&5&8&7\\
    3&3&4&1&2&7&8&5&6\\
    4&4&3&2&1&8&7&6&5\\
    5&5&6&7&8&1&4&2&3\\
    6&6&5&8&7&4&1&3&2\\
    7&7&8&5&6&2&3&1&4\\
    8&8&7&6&5&3&2&4&1
\end{array}.
\end{equation}

\begin{theorem}[Commutative automorphic loops of order $p^3$]\label{Th:p3}
Let $p$ be a prime. If $p=2$, let $O_2$, $\dots$, $O_4$ be as in Proposition \ref{Pr:n3}. If $p=3$, let $O_2$, $\dots$, $O_5$ be as in Proposition \ref{Pr:3}. If $p>3$, let $O_2$, $\dots$, $O_5$ be as in Proposition \ref{Pr:n3}. For every $i$, let $N_i\in O_i$ and $Q_i=F_p/N_i$.

There are precisely $7$ commutative automorphic loops of order $p^3$ up to isomorphism, including the three abelian groups $(\mathbb Z_p)^3$, $\mathbb Z_p\times\mathbb Z_{p^2}$, $\mathbb Z_{p^3}$. If $p$ is odd, the nonassociative commutative automorphic loops of order $p^3$ are precisely the loops $Q_2$, $\dots$, $Q_5$. If $p=2$, the nonassociative commutative automorphic loops of order $p^3$ are precisely the loops $Q_2$, $\dots$, $Q_4$, and the loop with multiplication table \eqref{Eq:8}.
\end{theorem}
\begin{proof}
The three abelian groups of order $p^3$ are certainly automorphic.

By a result of \cite{JKV3}, every commutative automorphic loop of odd order $p^k$ is nilpotent. Hence every nonassociative automorphic loop of odd order $p^3$ is in $\mathcal A_p$, by Lemma \ref{Lm:InAp}. The loops of $\mathcal A_p$ are classified up to isomorphism in Theorem \ref{Th:Class2}, and they include the abelian group $Q_1$ ($\cong\mathbb Z_p\times\mathbb Z_{p^2}$).

When $p=2$ we can proceed similarly, except that we have to account for the unique commutative automorphic loop \eqref{Eq:8} of order $8$ with trivial center.
\end{proof}

Theorem \ref{Th:p3} proves \cite[Conjecture $5.12$]{JKV2} and answers \cite[Problem 5.13]{JKV2}.

\end{document}